\documentclass{amsart}

\usepackage{amssymb,epsfig,newlfont,amsmath}
\usepackage{color}

\usepackage[french,english]{babel}
\usepackage[T1]{fontenc}

\newtheorem{theorem}{Theorem}[subsection]
\newtheorem{lemma}[theorem]{Lemma}
\newtheorem{proposition}[theorem]{Proposition}
\newtheorem{corollary}[theorem]{Corollary}

\newtheorem{definition}[theorem]{Definition}

\newtheorem{remark}[theorem]{Remark}
\newtheorem{remarks}[theorem]{Remarks}

\newtheorem*{example1*}{Example - the case $H=GL(4)$}
\newtheorem*{example2*}{Example - the case $H'=GL(2,D)$, $D$ a quaternion division algebra}

\catcode`\Ž=13
\catcode`\Ž=13
\catcode`\ˆ=13
\catcode`\=13
\catcode`\=13
\catcode`\‰=13
\catcode`\™=13
\catcode`\"=13
\catcode`\•=13
\catcode`\ž=13
\catcode`\=13
\catcode`\=13
\catcode`\Ï=13
\catcode`\Ë=13

\def Ž{\'e}
\def {\`e}
\def ˆ{\`a}
\def {\`u}
\def {\^e}
\def ‰{\^a}
\def ™{\^o}
\def "{\^{\i}}
\def •{\" {\i}}
\def ž{\^u}
\def {\c c}
\def Ï{\oe}
\def Ë{\`A}

\def\RM {\mathbb R}
\def\CM {\mathbb C}

\def\GM {\mathbb G}

\def\Hom{\mathrm{Hom}}
\def\Ker{\mathrm{Ker}}
\def\Coker{\mathrm{Coker}}
\def\card{\mathrm{card}}
\def\Ind{\mathrm{Ind}}
\def\NN{{\mathbf N}}
\def\adef{{\mathbb A_F}}
\def\H{H}
\def\G{G}
\def\Hplus{\H^+}

\def\T{T}
\def\TH{\T^\H}
\def\TG{\T^\G}
\def\Z{Z}
\def\ZH{\Z^H}
\def\ZG{\Z^G}
\def\Zun{\Z_1}
\def\ZO{\Z_0}
\def\Zplus{\Z^+}
\def\bs{\backslash}
\def\f{f}
\def\coh{{\mathbf H}}
\def\K{K}

\def\mun{^{-1}}
\def\cusp{\mathrm{cusp}}
\def\vf{\varphi}

\def\pni{\par\noindent}

\title[ Central morphisms and cuspidal representations]
{Central morphisms and Cuspidal automorphic Representations}

\begin{document}
\author{Jean-Pierre Labesse}
\thanks{This work was partly begun  during the stay of the second author in the fall term 2017 
at the School of Mathematics, Institute for Advanced Study, Princeton, and then pursued at the 
Max-Planck-Institute for Mathematics, Bonn; he gratefully acknowledges the funding at the IAS, 
provided by the Charles Simonyi Endowment, as well as the support at the MPIM.}
\address{Institut de Math\'ematique de Marseille, UMR 7373,
Aix-Marseille Universit\'e,
France}

\email{Jean-Pierre.Labesse@univ-amu.fr}

\author{Joachim Schwermer}
\address{Faculty of Mathematics, University Vienna, Oskar-Morgenstern-Platz 1, A-1090 Vienna, Austria  resp. Max-Planck-Institute 
for Mathematics, 
Vivatsgasse 7, D-53111 Bonn, Germany.}
\email{Joachim.Schwermer@univie.ac.at}

\date{}
 
\maketitle

\section{Introduction}

\subsection{Main theorem}
Let $F$ be a global field (of arbitrary characteristic) and denote by $\adef$ its ring of adles.
Let $\G$ and $\H$ be two connected reductive group defined over $F$
 endowed with  an $F$-morphism 
$f: \H\to\G$ such that the induced morphism $\H_{der}\to\G_{der}$ on the derived groups is a central isogeny.  
We study how automorphic representations behave, under the morphism induced between groups
of adŽlic points, via restriction and induction. 
We also discuss  similar statements for representations of groups over local fields.

Consider the restrictions to $f(\H(\adef))$
of a cuspidal representation $\pi$  of $\G(\adef)$; it splits into a possibly infinite sum of irreducible
representations of  $\H(\adef)$ and some of them may not be automorphic.  Conversely given 
a cuspidal representation $\sigma$  of $\H(\adef)$ it is not always possible to find it in the
restriction of some cuspidal representation $\pi$ of $\G(\adef)$. 
Our main result is:

\begin{theorem} \label{goal} 
Given any irreducible cuspidal representation $\pi$ of $\G(\adef)$
its restriction to $f(\H(\adef))$ contains a cuspidal representation $\sigma$ of $\H(\adef)$.
Conversely,   assuming moreover that  $f$ is an injection, 
 any irreducible cuspidal representation $\sigma$ of $\H(\adef)$ 
appears in the restriction of some cuspidal representation $\pi$ of $\G(\adef)$.
\end{theorem}

Experts in the theory of automorphic forms  have  expected such a natural result, already known in some cases.
In fact this is in agreement with Langlands functoriality conjectures which relate
local or automorphic representations for a group $\G$
to elements in the first cohomology set  $H^1(W_F,\check\G)$ of Weil groups with value 
in the dual group $\check\G$ and, in particular, with lifting results established in \cite{Lab}
for the map $H^1(\check f):H^1(W_F,\check\G)\to H^1(W_F,\check\H)$
when $f$ is injective.

Local results,  quite elementary when the characteristic of the field is zero and
already known in general to some extent (see for example \cite{T}), are given here  
for the sake of completeness. In the global case,  we do not know of any published reference
 except when  $H = SL(n)$,  $\G = GL(n)$.
For this pair of groups, Theorem \ref{goal}  (or rather its reformulation as in \ref{main})
is claimed in  \cite[Sect. 3]{LS}   if  $F$  is a  number field.
Unfortunately, as was pointed out by Laurent Clozel, the proof given in \cite{LS},
which generalizes the argument given in \cite{LL}  for $n=2$,
does not apply for arbitrary $n$
since we implicitly assumed the validity of the local-global principle for the $n$-th powers  in  $F$, 
which may fail\footnote {%
The Grunwald-Wang Theorem \cite[Chap. X, Thm. 1]{AT}
computes the obstruction group to this local-global principle and, in fact, this group can be
non-trivial. This only happens in very special cases: in particular $8\vert n$
is among the necessary conditions.}.
The argument is corrected here.

 \subsection{Organization of the paper}
In Section \ref{cliff}, generalizing Clifford's theory for finite groups,
we consider  a pair 
$B \subset A$ of a locally compact groups where $B$ is a closed invariant subgroup of $A$
 such that $B\bs A$  is abelian and compact and we 
analyze, under a suitable finiteness condition, the interplay via restriction and induction between irreducible unitary 
representations $\pi$ of $A$ and $\sigma$ of $B$. 
   Results of Section \ref{cliff} are used in Section  \ref{loc} to investigate 
extension, induction and restriction of irreducible unitary representations between pairs of groups of points 
over local fields of arbitrary characteristic for  pairs of connected reductive groups $G$ and $H$ 
as above.  The local results are summarizd in \ref{local} and \ref{bija}.
Similar questions are studied in section
 \ref{cusp}  for cuspidal automorphic representations of groups of points
over adles of global fields of arbitrary characteristic
and our main results are Theorems \ref{cuspa} and \ref{cuspb}. They imply in particular
the above Theorem \ref{goal}.
We rely on structural results taken care of in Section \ref{setting} or, 
by a very different approach, in the appendix provided by B. Lemaire. 
We conclude the paper with a new multiplicity formula.

\subsection{Acknowledgements}
 We thank Bertrand Lemaire for providing us with the appendix. We also thank Gaetan Chenevier 
 who drew our attention to his unpublished note  \cite{Ch} and for pointing out a mistake 
 in a previous version of the present note. Finally we are grateful to Guy Henniart who suggested, 
 following his paper \cite{He}, to consider not only inclusions 
 $\H\subset\G$ but also morphisms $\H\to\G$ inducing central isogenies between derived groups.

\section{Variation on a theme by Clifford}\label{cliff}

In this section we  establish a variant of Clifford's theory for finite groups \cite{CR} (already used  implicitly 
 in \cite{LL} or  explicitly in \cite{He} but in a slightly different context).
This is elementary but,
not knowing of any reference valid in our setting, it is given with some details  for the convenience of the reader.   
In particular, we prove\footnote{%
The reader may wonder why we give a proof of such a result since
 there are many references for instances where  Frobenius reciprocity is known to hold. In fact, 
 for admissible representations of  reductive groups over non archimedean local fields
  Frobenius reciprocity is well known and could be used in certain sections below where we deal with this specific case.
  Nevertheless more general groups and different kinds of representations will occur
 and we did not find any reference for the form
  we need: for example Moore's result (Section 4 of \cite{MCC}), which is the closest to our needs we could find,
   applies only to  finite dimensional representations; similarly
 Mackey's quite general theorems (e.g.\! Theorem 5.1 of \cite{Mc}) does not seem to be of any help 
 since the representations we are dealing with
 may show up with measure zero in the spectral decomposition of the right regular representations
 for the groups we study.}
  a form of Frobenius reciprocity in \ref{frobenius}.

\subsection{Notation}
 It is understood that unitary representations are strongly continuous and characters 
  (i.e.\! one dimensional representations) are unitary.
By abuse of notation we shall often denote by the same symbol a quotient and a set of representatives for its elements.

Let $A$ be a locally compact group and $B$ a closed  subgroup such that 
 $B\bs A$ has an $A$-invariant measure.
Let $\pi$ be an irreducible unitary representation of $A$ and
$\sigma$ an irreducible unitary representation of $B$.  
Given  $g\in A$ we denote by $\sigma^g$ the  representation of $B$ 
defined by
$$\sigma^g(x)=\sigma(gxg^{-1})\,\,.$$
Let   $\rho$ be the induced representation
 $$\rho=\mathrm{Ind}_B^A\sigma\,\,.  $$
 We denote by $<v, w>_\sigma$ the scalar product of two vectors $v$ and $w$
 in the space $V_\sigma$ of the representation $\sigma$.
 Recall that $V_\rho$, the space of $\rho$, is the set of classes of measurable functions
 from $A$ to $V_\sigma$  (up to equality almost everywhere),
such that $ f(hg)=\sigma(h)f(g)$ and that are square integrable on $B\bs A$.

\subsection{A first finiteness assumption} 

Assume $B\bs A$ is  of finite volume.

 \begin{lemma}\label{frob}  
There is an injective map 
$$\mathrm{Hom}_B(\pi\vert_B,\sigma)\to\mathrm{Hom}_A(\pi,\mathrm{Ind}_B^A\sigma)\,\,.$$
If $\pi\vert_B$ and $\pi'\vert_B$  have a common constituent $\sigma$ then
$\pi$ and $\pi'$ both occur in $\rho=\mathrm{Ind}_B^A\sigma\,\,.  $
 \end{lemma}
 
 \begin{proof}
  Consider an  element $\Psi\in  \mathrm{Hom}_B(\pi\vert_B,\sigma)$ and $w\in V_\pi$. The function
$$\varphi_w:g\mapsto\Psi(\pi(g)w)\qquad\qquad\hbox{for}\quad g\in A$$ defines a vector in   $V_\rho$. In fact
 this is a continuous function which satisfies the required functional equation and
 whose square norm  $$g\mapsto||\varphi_w(g)||^2:=<\varphi_w(g),\varphi_w(g)>_\sigma\le
 ||\Psi||^2<w,w>_\pi
 $$
 is bounded and hence integrable since $B\bs A$ is of finite volume.  The map
$$\Phi:w\mapsto \varphi_w$$ defines an element in $\mathrm{Hom}_A(\pi,\rho)$.
The assignment $\Psi\mapsto\Phi$ is obviously injective. 
The second assertion follows immediately.
\end{proof}

Assume from now on that $A$ is unimodular and
$B$ is an invariant closed subgroup. The quotient group $C=B\bs A$ is also assumed to be abelian   compact
 and endowed with the normalized Haar measure i.e.\! such that $\mathrm{vol}(C)=1$.
Let $X$ be the discrete group of characters of $C$.
 We observe that if $\pi$ occurs in $\rho  =\mathrm{Ind}_B^A\sigma\,\,$ then, given $\chi\in X$,
 the representation
 $\pi\otimes\chi$ also occurs in $\rho$ with the same multiplicity.

\begin{proposition}\label{cliffb} 
Given $\pi$ and $\pi'$ two irreducible unitary representations of $A$ whose restrictions
to $B$ have a constituent $\sigma$ in common, then there exist a character $\chi\in X$ such that
$$\pi'\simeq\pi\otimes\chi\,\,.$$
The representation $\rho$ is an Hilbert direct sum of representations of the form  $\pi\otimes\chi$.
\end{proposition}

\begin{proof}
Since both the restrictions of
 $\pi$ and $\pi'$  to $B$ have $\sigma$ as a constituent in common,  Lemma~\ref{frob} shows
  they both occur in $\rho$. Let us denote by $V_\pi$ the space of the representation $\pi$.
 Let $\Psi$ be a non-trivial intertwining operator in $$\Hom_B(\pi\vert_B,\sigma)$$
 and consider for $w\in V_\pi$ the function
 $$\varphi_w:g\mapsto\Psi(\pi(g)w)\,\,.$$
The closed subspace generated by the functions $\varphi_w\chi$, where
$w$ varies in $V_\pi$ and $\chi$ varies in $X$,
is the space of a subrepresentation $\rho'$ of $\rho$, generated by a set of subrepresentations isomorphic to
 $\pi\otimes\chi$.
 Let  $f$ be a  function from $A$ to $V_\sigma$
 that belongs to the orthogonal $\rho''$ of  $\rho'$. 
We have to show that $f=0$. Let us denote by 
 $<\varphi,f>_\rho$ the scalar product of two functions $\varphi$ and $f$ in the space of $\rho$.
 By hypothesis
$$<\varphi_w{\chi},f>_\rho=\int_{C} {\chi(g)}<\varphi_w(g),f(g)>_\sigma\,d{\dot g}=0$$
for all $w\in V_\pi$ and all $\chi\in X$.
This implies that $ <\varphi_w(g),f(g)>_\sigma=0$ for almost all $\dot g\in C$ and all $w \in V_\pi$.
Now $w\mapsto\varphi_w(g)$ is an intertwining operator $\Psi^g$ between $\pi\vert_B$ and $\sigma^g$,
a representation of $B$ in $V_\sigma$
which is irreducible;  the image of $\Psi^g$ equals $V_\sigma$  and necessarily $f(g)=0$ for almost all $g$.
  \end{proof}
  
  \subsection{A second finiteness assumption}
  
We denote   by $A(\sigma)$  the subgroup of $A$
(containing $B$) of $g\in A$ such that $\sigma^g\simeq\sigma$ and by 
 $X(\pi)$ the subgroup of $\chi\in X$ such that $\pi\otimes\chi\simeq\pi$.
We shall now moreover assume that $X(\pi)$ is finite.

\begin{proposition}\label{cliffa}
Let $\pi$ be an irreducible unitary representation  of $A$ such that $X(\pi)$ is finite.
 Its restriction $\pi\vert_{ B}$  is a finite direct sum of
irreducible unitary representations of $B$. Let $\sigma$ be an irreducible constituent of $\pi\vert_{ B}$.
The vector space  $$V= \Hom_B(\pi\vert_B,\sigma)$$
 is  of finite dimension, say $m$.
All other constituents
are conjugates under $A$ of $\sigma$ and
$$\pi\vert_{ B}\simeq \bigoplus _{\dot g\in A/A(\sigma)} V\otimes\sigma^g$$
where $B$ acts trivially on $V$.
The algebra $\mathcal I(\pi)$, of intertwining operators for $\pi$ restricted to $B$, has a basis indexed
by $X(\pi)$ and
 $$\mathrm{dim}(\mathcal I(\pi))=\card(X(\pi))=m^2\times\card(A/A(\sigma)) \,\,.$$
\end{proposition}

\begin{proof}  For $\chi\in X(\pi)$
choose  a non-trivial intertwining operator $U_\chi$ between $\pi$ and $\pi\otimes\chi$.
 According to Schur's lemma, the operator $U_\chi$ is well defined up to a scalar.
Consider $I\in\mathcal I(\pi)$ and $\chi\in X$, then the operator
$$I_\chi=\int_{\dot g\in C}\overline{\chi(g)}.\pi(g)^{-1} I\pi(g)\,\,d{\dot g}$$
 is a scalar multiple of $U_\chi$ for $\chi\in X(\pi)$ and is zero if $\chi\notin X(\pi)$. 
 Fourier inversion shows that
$$I=\sum_{\chi\in X(\pi)} I_\chi=\sum_{\chi\in X(\pi)} c_\chi U_\chi$$ with $c_\chi\in\CM$.
This implies $\mathrm{dim}(\mathcal I(\pi))=\card(X(\pi))$. 
By assumption $\sigma$ is an irreducible constituent of $\pi\vert_{ B}$.
The closed subspace generated by the isotypic components of
$\sigma$ and its $A$-conjugates  is an $A$-invariant subspace
of $V_\pi$, equal to $V_\pi$ since $\pi$ is irreducible.  Hence  $\pi\vert_{ B}$ 
is isomorphic to a finite sum of irreducible representations of $B$ that are $A$-conjugates 
of $\sigma$. 
\end{proof}

The group  $X(\pi)$ is of course finite when $C$ is finite but there are many other instances of it, in particular
  when dealing with admissible representations (see \ref{fini} and \ref{finiadele} below).

One should note that the algebra $\mathcal I(\pi)$ may not be isomorphic to the group algebra $\CM[X(\pi)]$. 
This is the case when $m\ge2$.
An example occurs in the study of inner forms of $SL(2)$ (cf. \cite{LL}) where one may have 
$A(\sigma)=A$ while $X(\pi)$ is an abelian group
of order 4 but $m=2$ and $\mathcal I(\pi)=M(2,\CM)$ the algebra of $2\times2$ matrices.
Further examples are given in \cite{HS}.

Consider the  subgroup $B(\pi)$ of $g\in A$
such that $\chi(g)=1$ for all $\chi\in X(\pi)$. If $X(\pi)$ is finite,  $B(\pi)$ is of  index $\card(X(\pi))$ in $A$
and we have  the following inclusions
$$B\subset B(\pi)\subset A(\sigma)\subset A\,\,.$$

\begin{corollary} If $X(\pi)$ is finite the representation $\sigma$ of $B$ can be extended to a representation
$\tilde\sigma$ of $B(\pi)$ in the same space.
\end{corollary}

\begin{proof}
 Proposition~\ref{cliffa} applied to the pairs $(A,B)$ and $(A,B(\pi))$ 
tells us that the dimension of the intertwining algebra for $\pi\vert_{B(\pi)}$ and $\pi\vert_B$ are both equal
to $\card(X(\pi))$
and hence the irreducible constituents of $\pi\vert_{B(\pi)}$ remain irreducible when
restricted to $B$. \end{proof}

\begin{proposition} \label{frobenius} Let $B \subset A$ be a pair of a locally compact groups where 
$B$ is a closed invariant subgroup of $A$ such that $B\bs A$  is a compact abelian group. 
 Let $\pi$ and $\sigma$ be irreducible unitary representations of $A$ and  $B$ respectively.
Assume that the group $X(\pi)$, of characters of $B\bs A$ such that $\pi\otimes\chi\simeq\pi$,
 is finite. Then Frobenius reciprocity holds: the natural map
$$Frob:\Hom_B(\pi\vert_B,\sigma)\to\Hom_A(\pi,\mathrm{Ind}_B^A\sigma)\,\,.$$
is an isomorphism. 
\end{proposition}
\begin{proof}
We have seen that $\sigma$ can be extended to a representation  $\tilde\sigma$ of $B(\pi)$.
Since $A/B(\pi)$ is finite all functions in the space of $\Ind_{B(\pi)}^A\tilde\sigma$ are continuous
and  evaluation at the origin yields
the Frobenius reciprocity, i.e.\! the following map is a bijection:
$$\Hom_{B(\pi)}(\pi\vert_{B(\pi)},\tilde\sigma)\to\Hom_A(\pi,\Ind_{B(\pi)}^A\tilde\sigma)\,\,.
\leqno{(a)}$$
On the other hand,  there is an isomorphism
$$\Hom_{B}(\pi\vert_{B},\sigma)\to\Hom_{B(\pi)}(\pi\vert_{B(\pi)},\tilde\sigma)\,\,.\leqno{(b)}
$$
Now $\tilde\sigma$ injects in $\Ind_B^{B(\pi)}\sigma$: in fact
the map
$$w\mapsto f_w\qquad \hbox{with}\qquad f_w(x)=\tilde\sigma(x)w$$
is an intertwining operator since
$$\rho(y)f_w(x)=f_w(xy)=
\tilde\sigma(xy)w=f_{\tilde\sigma(y)w}(x),$$
and it follows form \ref{cliffb} that
$\Ind_B^{B(\pi)}\sigma$ is the Hilbert direct sum of the $\tilde\sigma\otimes\nu$
where $\nu$ runs over characters of $B(\pi)/B$ while $\Ind_{B(\pi)}^A\tilde\sigma$
is a multiple of $\pi$. This implies that 
$$\Hom_A(\pi,\Ind_{B(\pi)}^A(\tilde\sigma\otimes\nu))=\Hom_A(\pi,(\Ind_{B(\pi)}^A\tilde\sigma)\otimes\chi)=0\leqno{(c)}$$
unless $\chi$, which is any extension of $\nu$ to $A$,
belongs to $X(\pi)$.
Now induction by stages shows that
$$\Ind_{B}^A\,\sigma=\widehat\bigoplus_\nu\Ind_{B(\pi)}^A(\tilde\sigma\otimes\nu)$$
where $\nu$ runs over characters of $B(\pi)/B$ and (c) implies
$$\Hom_A(\pi,\rho)=\Hom_A(\pi,\Ind_{B}^A\,\sigma)=\Hom_A(\pi,\Ind_{B(\pi)}^A\tilde\sigma)\leqno{(d)}.$$
In view of (a), (b) and (d)
the proof is complete.
\end{proof}

\section{The groups in question}\label{setting}

Let $k$ be a  field. 
Let $\H$ and $\G$ be  two 
connected algebraic groups over $k$ with a morphism $\f:\H\to\G\,\,.$
Let $\ZG$ denote the center of $\G$ and $\ZH$ the center of $\H$.  
Let $\Z$ be the connected component of $\ZG$; this is a torus.

\subsection{Some crossed modules}
We shall assume that the natural morphism $\Z\times\H\to\G$
 is a central map which means that it is surjective and its kernel an abelian group scheme 
in the center (see Appendix A).
 This is equivalent to ask that the morphism induced between the derived subgroups 
 $$\f_{der}:\H_{der}\to\G_{der}$$
 is a central isogeny. 
 This is also equivalent to asking that the induced map 
  $$\f_{ad}:\H_{ad}\to\G_{ad}$$ between the adjoint groups is an isomorphism.
 The last isomorphism shows that $\G$ acts on $\H$ by conjugacy
 and this implies that the complex $[H\to\G]$ is a crossed-module. 
 We refer the reader to \cite[Chap.~1]{LBC} or \cite[Appendix~B]{Mi}
 for this concept.
 The particular case where $\H=\G_{sc}$  is the simply connected cover of the derived group has been extensively 
studied in \cite{LBC}.

 \begin{lemma} \label{quasi} Let   $\TG$ be  a maximal torus in $\G$ and let $\TH$ be its inverse image in $\H$.
 The map between complexes $$[\TH\to\TG]\to[H\to\G]$$
 induces a quasi-isomorphism between complexes of points over the separable closure.
 \end{lemma}
 
 \begin{proof} Let $k^{sep}$ denote the separable closure of $k$. We want to prove that
 $$[\TH(k^{sep})\to\TG(k^{sep})]\to[H(k^{sep})\to\G(k^{sep})]$$
 is a quasi-isomorphism. In particular we need to compute the kernel and cokernel of the map $\f^{sep}$.
But then we are dealing with split groups and split tori.
Since the unipotent subgroups and Weyl groups
are isomorphic  \cite[Th\'eor\`eme (2.20), page 260]{BT} and using Bruhat decomposition,
we are left with the kernel and cokernel of the induced map between the tori.
  \end{proof}

 \subsection{Crossed modules over local  fields}\label{crossl}
 
 In this subsection  $F$ is a  local field; by this we mean
archimedean or non archimedean local field as well.  
As a convention for the Galois cohomology with values in complexes $[B\to A]$ we take $A$ in degree 0.
  This is the convention used in \cite{LBC}.

We denote by $\Hplus$ the subgroup of $\G(F)$ generated by $\f(\H(F))$ and $ \ZG(F)$. 
 The reader is warned that although  $\Hplus$ is a Lie group when $F$ is archimedean
and a totally disconnected group when $F$ is non archimedean,
it is not in general the group of points of an algebraic  reductive group over $F$.

\begin{proposition}\label{coloc} The group
$\Hplus$ is an invariant subgroup in $\G(F)$. The quotient $G(F)/\Hplus$  is  abelian and compact; 
it is even finite for local fields of characteristic zero.  
\end{proposition}

\begin{proof}   Replacing if necessary $\H$ by $\H\times\Z$ which is again reductive and connected,
we may assume that the map $\f$ is surjective. Then
 it is enough to prove that, in such a case, $\f(\H(F))\bs \G(F)$ is abelian, compact and 
even finite for local fields of characteristic zero. 
Since quasi-isomorphisms between complexes of points on the separable closure compatible
with Galois action
induce isomorphisms in Galois cohomology \cite[Proposition 1.2.2]{LBC},
Lemma \ref{quasi} implies that
 the map
$${\coh}^0(F,[\TH\to\TG])\to {\coh}^0(F,[\H\to\G])$$
is an isomorphism and hence ${\coh}^0(F,[\H\to\G])$ is abelian.
 One has  an exact sequence 
 $$
1\to\f(\H)\bs\G\to\coh^0(F,[\H\to\G])\to\coh^1(F,\H).
$$
In particular  $\f(\H)\bs \G$ is an abelian subgroup of finite index in $\coh^0(F,[\H\to\G])$.
There is an exact sequence
 $$
1\to\f(\TH)\bs\TG\to\coh^0(F,[\TH\to\TG])\to\coh^1(F,\TH)\,\,.
$$
Since  $\coh^1(F,\TH)$ is finite it remain to observe that  $\f(\TH)\bs\TG$ is compact when $\f$ is surjective
 (see for example Lemma~\ref{apptor} in Appendix A). 
It is finite for local fields of characteristic zero 
\end{proof}

For an alternative argument independent of Galois hypercohomology
see \ref{appcomp} in Appendix A.

\begin{remark}  In the case of a central isogeny $\H \to \G$ for groups over a non-archimedean 
local field this result was stated (without proof) and used in \cite{Si}.
\end{remark}

 \subsection{Crossed modules over  global fields}\label{crossg}

In this subsection $F$ is a global field.
  We shall use the notation of \cite{KS} and \cite{LBC} for adelic cohomology. The reader should be aware that
  the degree conventions for hypercohomology of complexes are not the same in
  these references: namely
 $\coh^0(\star,B \to A)$  in \cite{LBC} is  $\coh^1(\star,B\to A)$  in \cite{KS}.  We shall use the
 convention of  \cite{LBC}.

\begin{lemma}\label{co} Assume the morphism $\f: H \to G$ is surjective.
Then $\coh^0(\adef/F,[\H\to\G])$ is compact.
\end{lemma}

\begin{proof}    The quasi-isomorphism $[\TH\to\TG]\to[\H\to\G]$ implies  isomorphisms in cohomology.
Hence it is equivalent to prove that  $$\coh^0(\adef/F,[\TH\to\TG])$$ is compact. However, this 
is one of the statements in Lemma C.2.D, page 153, in  \cite{KS} (up to the  shift in degree explained above).
 Although this reference is written for number fields the proof extends verbatim to the case of arbitrary global fields. Namely,
one has an exact sequence
$$1\to D\to\coh^0(\adef/F,[\TH\to\TG])\to \coh^1(\adef/F,\TH)$$
where $$D=\Coker[\coh^0(\adef/F,\TH)\to\coh^0(\adef/F,\TG)]$$
 is compact  if $\f$ is surjective while $\coh^1(\adef/F,\TH)$ is finite.  
   \end{proof}
  
  \begin{lemma}\label{cor} Assume $\f$ is surjective.
Then $\G(F)\f(\H(\adef))\bs\G(\adef)$ is compact.
\end{lemma}

\begin{proof}   Let us denote by $\K$ the complex $[\H\to\G]$.
The following  diagram
$$\begin{matrix}
&&\H(F)&\to &\H(\adef)&\to &\coh^0(\adef/F,\H)&&\cr
&&\downarrow&&\downarrow&&\downarrow&&\cr
&&\G(F)&\to&\G(\adef)&\to&\coh^0(\adef/F,\G)&&\cr
&&\downarrow&&\downarrow&&\downarrow&&\cr
&& \coh^0(F,\K)&\to&\coh^0(\adef,\K)&\to&\coh^0(\adef/F,\K)&\to&\Ker^1(F,\K)\cr
&&\downarrow&&\downarrow&&&&\cr
\Ker^1(F,\H)&\to& \coh^0(F,\H)&\to&\coh^0(\adef,\H)&&&&\cr
\end{matrix}$$
 is commutative with exact lines and columns.
  Now \ref{co} and the finiteness  of $$\Ker^1(F,\K)\simeq \Ker^1(F,[\TH\to\TG])$$
  (cf. \cite{KS})
  imply that 
 $$\Coker[\coh^0(F,[\H\to\G])\to\coh^0(\adef,[\H\to\G])]$$
 is also compact. 
 Thanks to the finiteness of $\Ker^1(F,\H)$ (cf.  \cite[Prop. 1.7.3]{LBC} for number fields,
 which rephrases results of Kottwitz
  and   \cite[Thm 1.3.3]{Con} for function fields. The latter one relies on  \cite{Ha})
 the image of $\G(\adef)$ in this cokernel is  up to a finite subgroup
 isomorphic to the quotient $$\G(F)\f(\H(\adef))\bs\G(\adef)$$
 and hence this quotient is also compact.
  \end{proof}

We now return to the general case where $f: H \to G$ need not be surjective.

\begin{proposition}\label{coad} Let $\Hplus: = \ZG(\adef)\G(F) \f(\H(\adef))$.
The quotient $\Hplus\bs\G(\adef)$ is an abelian compact group.
\end{proposition}

\begin{proof}   Replacing
if necessary $\H$ by $\H\times\Z$ this is a consequence of Lemma \ref{cor}.
 \end{proof}   

For an alternative argument independent of adlic hypercohomology
see \ref{appcompb} in Appendix A.

\section{A first application of Clifford's theory: the local case}\label{loc}

In this section  $F$ is a local field.  
{Some aspects of what follows have been observed by various authors}
(see in particular  \cite{GK},  \cite{He}, \cite{HS}, \cite{LL}, \cite{LS}, \cite{Si} and \cite{T}).

\subsection{The basic results}

Consider $\G$ and $\H$ with a map $\f:\H\to\G$ over $F$ inducing a central isogeny of their derived groups
and consider $\Hplus$ the subgroup of $\G(F)$ generated by $\f(\H(F))$ and $ \ZG(F)$.  
We denote by $N$ the kernel of the map  $\f:\H\to\G$.

\begin{lemma}\label{fini} 
The quotient $\Hplus\bs \G(F)$ is abelian compact.
If $\pi$ is an irreducible unitary representation of $\G(F)$ then  the group $X(\pi)$ of characters $\chi$ of $\Hplus\bs \G(F)$
such that $\pi \otimes \chi \cong \pi$ is finite. 
\end{lemma}

\begin{proof}  
We apply the results of section~\ref{cliff} to $A=\G(F)$  and $B=\Hplus$.
 The assertions are obvious when $C=\Hplus\bs \G$ is finite which is the case
for local fields of zero characteristic according to Proposition \ref{coloc}. For
 non archimedean fields of arbitrary characteristic we appeal again to Proposition \ref{coloc} or \ref{appcompb}
  for the first statement.
 The finiteness of $X(\pi)$ is known for admissible irreducible representations
  (\cite{He}, \cite{Si}). To conclude we recall that the subspace of smooth vectors in an
irreducible unitary representation of the group of points of a connected
reductive group group over a non archimedean local field is admissible (\cite{Be}).
 \end{proof}

Let $F$ be a non archimedean local field,  and assume $\G$ is a  quasi-split
connected reductive group,
split over an unramified extension. Choose an hyper-special maximal compact subgroup $K\subset G(F)$.
 We say that a representation $\pi$ of $\G(F)$  is unramified if  the operator 
 $\pi(K)$ fixes a non zero vector.  

\begin{lemma}\label{unr} If  $\pi$ is unramified, all elements in $X(\pi)$ are also unramified.
\end{lemma}

\begin{proof} Choose an Iwahori subgroup $I\subset K$; then there is a 
 unique Borel subgroup $P_0\subset G$ with Levi decomposition
 $P_0=T\ltimes U$  such that $$I=(T(F)\cap I)(U(F)\cap I)(\overline U(F)\cap I) $$
 where $\overline U$ is the opposite unipotent subgroup.
An unramified  representation $\pi$ is the spherical  subquotient of a principal series
representation obtained by parabolic induction of a character  $\lambda$ of $T(F)$ which is  trivial on $T(F)\cap I$.
A character $\chi\in X(\pi)$ defines by restriction a character   $\tilde\chi$ of $T(F)$.
 The representation $\pi\otimes\chi$  is a  subquotient of the principal series
representation obtained by parabolic induction of  $\lambda\tilde\chi$.
But  since $\pi\simeq\pi\otimes\chi$ one has
$$\lambda\tilde\chi=s(\lambda)\qquad\hbox{for some $s$ in the Weyl group}\,\,.$$
This shows that $\tilde\chi=s(\lambda)\lambda\mun$ is trivial on $T(F)\cap I$. 
Then $\chi$ must be trivial on the subgroup generated by
$\f(\H(F))$ and $T(F)\cap I$. Now $f(H(F))\supset U(F)\supset U(F)\cap I$ and similarly for $\overline U$.
Hence $\chi$ is trivial on $I$. Denote by $K'$ the hyper-special subgroup in $H(F)$ such that
$f(K')\subset K$. Any $s'\in W'$ has  a representative $w_{s'}\in H(F)\cap K'$. 
The Weyl group $W'$ of $H(F)$ maps bijectively via $f$ onto the Weyl  group $W$ of $G(F)$
and hence any $s\in W$ has  a representative $$w_s=f(w_{s'})\in f(H(F))\cap K\,\,.$$ Since the $w_s$ and $I$ generate $K$
the character $\chi$ is trivial on $K$.  \end{proof}

\begin{proposition} \label{local}
Given an irreducible unitary representation $\pi$ of $\G(F)$ its restriction to $\H(F)$ is  a direct sum of finitely many
irreducible unitary representations that are $\G(F)$-conjugate.
Conversely, any irreducible unitary representation
of $\H(F)$ trivial on $N$
occurs in the restriction of some $\pi$ and all such irreducible representations are of the form $\pi\otimes\chi$ with
$\chi\in X$ the group of characters of $\G(F)/\Hplus$.
\end{proposition}

\begin{proof}
We  first restrict $\pi$ to $\Hplus$.  In view of Lemma  \ref{fini} we may use Proposition \ref{cliffa} with
  $A=G(F)$ and $B=\Hplus$. Hence
this restriction is a direct sum of finitely many
irreducible unitary representations of $\Hplus$ that are conjugate under $\G(F)$.
Then restriction from $\Hplus$ to $N\bs\H(F)$ preserves irreducibility.

Conversely, consider a representation $\sigma$ of $\H(F)$ and $\omega$ a character of  $ \ZG(F)$
 such that its restriction to $\ZH(F)$ is the character with which $\ZH(F)$ acts via  $\sigma$.  One can extend $\sigma$ to a 
representation  $\sigma^+$ of  $\Hplus$ and 
 then  induce this representation from $\Hplus$ to $\G(F)$.  According to Proposition \ref{cliffb} 
 this is a sum of  representations  of the form $\pi\otimes\chi$ with
 $\pi$ irreducible and
$\chi\in X$ the group of characters of $\Hplus\bs\G(F)$. The restriction of $\pi$ to $N\bs\H(F)$ contains $\sigma$ 
according to Proposition~\ref{frobenius}.
\end{proof}

We  observe that if $\G$ and $\H$ are quasisplit, and if $\pi$ is generic (i.e.\!  has a Whittaker model for some
character of the unipotent radical of a chosen Borel subgroup)
the restriction $\pi\vert_{\f(\H(F))}$ is multiplicity free (i.e.\! $m=1$
in the notation of  Proposition~\ref{cliffa}) 
as follows from Proposition~\ref{frobenius} using
the uniqueness of Whittaker models and the compatibility
of Whittaker models with induction.

\subsection{Two equivalence relations}

\begin{definition}   We say that two irreducible unitary representations 
$ \sigma$ and $ \sigma'$ of $\H(F)$ are in the same
``$\G(F)$-packet''  
 if there exists an element $g\in\G(F)$ such that  $\sigma^g \cong \sigma'$.
We denote by $\mathcal{A}_{\G}(\H)$ the set of $\G(F)$-packets of irreducible unitary representations of $\H(F)/N$.
\end{definition}

We  observe  that
$\G(F)$-packets coincide with $L$-packets  when $\H=SL(n)$  and $\G=GL(n)$
and for compatible inner forms as well. In general   $L$-packets  should be
unions of $\G(F)$-packets since adjoint conjugacy is a special case of stable conjugacy.

\begin{definition}We define two irreducible unitary representations $\pi $ and $\pi'$ of $\G(F)$ to be 
$\mathcal{E}_{\H}$-equivalent  if there exists a character $\mu$ of $\G(F)/\Hplus$ such that $\pi \otimes \mu \cong 
\pi'$.  
We denote by $\mathcal{E}_{\H}(\G)$ the  corresponding set of equivalence classes. 
\end{definition}

Now,  all elements in the $\mathcal{E}_{\H}$-class of some $\pi$ have equivalent restrictions to $\f(\H(F))$ and all components 
of the restriction
  belong to the same $\G$-packet. 
Let  $R$  be the map
  which assigns to an $\mathcal{E}_{\H}$-equivalence class 
  represented by $\pi$ the $\G(F)$-packet  of components $\sigma^g$ of the restriction of
  $\pi$ to ${{ \H(F)}}$.   
 The above Propositions and remarks can be summarized as 
  
  \begin{proposition} \label{bija}   The map 
  $R : \mathcal{E}_{\H}(\G) \to \mathcal{A}_{\G}(\H)$
   is a bijection. 
\end{proposition}

\section{Second application: the case of cuspidal  representations.}\label{cusp}

 Now $F$ is a global field and
 we examine how cuspidal automorphic representation behave under restriction and induction.
 By cuspidal representation we understand an irreducible unitary automophic representation ocuring in the cuspidal
 spectrum.  For a definition of these objects over fields of arbitrary characteristics
we refer  the reader to \cite{MW}.   We consider two connected reductive groups with a  
map $\f:\H\to\G$ over  some global field $F$   inducing a central isogeny of their derived groups.

\subsection{The key construction}

We have introduced in subsection~\ref{crossg} the subgroup
$$\Hplus: = \ZG(\adef)\G(F)\f( \H(\adef))$$ in $\G(\adef)$.
 According to Propositions \ref{coad} or \ref{appcompb}
 the quotient $\Hplus\bs\G(\adef)$ is abelian and compact.
 
   \begin{lemma}\label{finiadele} 
  Let $\pi$ be an automorphic representation of $\G(\adef)$. Then the group $X(\pi)$ of characters
 of $\Hplus\bs\G(\adef)$ such that $\pi\otimes\chi\simeq\pi$ is finite.
\end{lemma}

\begin{proof} 
   We observe that  $\Hplus$  contains the product over all places $v$ of groups $\Hplus_v$  
  generated by $\f(\H_v)$ and $ \ZG_v$.    Let $\pi$ be an  automorphic representation of $\G(\adef)$. Thanks to 
    Lemma \ref{fini} and \ref{unr}
    we know
 there is  a compact open subgroup $K_f$ of the finite ad\`eles
 on which any $\chi\in X(\pi)$ is trivial.
 Recall that $\Hplus_\infty$ is of finite index in $\G_\infty$ when $F$ is a number field.
 In all cases $ K_f.\Hplus$ is an open subgroup of finite index in $\G(\adef)$
 on which any $\chi$ such that $\pi\otimes\chi\simeq\pi$ is necessarily trivial, hence $X(\pi)$ is finite.
  \end{proof} 
  
Denote by $\NN$ the kernel of the map $\f_\adef:\H(\adef)\to\G(\adef)\,\,.$
This is a subgroup in the center of $\H(\adef)$
and we may identify $\NN\bs\H(\adef)$ with $\f(\H(\adef))$.
Let $\Zplus:= \ZG(\adef)\G(F)$.
Observe that
$$\Zplus/\G(F)= \ZG(\adef) \G(F)/\G(F)
=  \ZG(\adef)/  \ZG(F)\,\,.$$
  Let 
$$ \Zun^+:  \Zplus \cap \f(\H(\adef)) \qquad\hbox{and} \qquad \Zun=\f\mun(\Zun^+)\,\,.$$
$\Zun$  is a  closed subgroup  in $\H(\adef)$ that contains and
normalizes  $\H(F)$.  Let  $$\Gamma^+=\G(F)\cap \Zun^+=\G(F)\cap\f(\H(\adef))
\qquad\hbox{and}\qquad\Gamma=\f\mun(\Gamma^+)
\,\,.$$
The subgroup  $\Gamma$ in $\H(\adef)$ contains $\NN.\H(F)$ and
$\Zplus_1/\Gamma^+\simeq\Zun/\Gamma\,\,.$
 Thus,  a unitary character $$\omega:  \ZG(\adef)/ \ZG(F) \to \CM^{\times}$$
 defines a character of $\Zplus$,  again denoted $\omega$, and
 we obtain by restriction a character $\omega_1^+$ on  $\Zun^+$ trivial on $\Gamma^+$.
  Observe that conversely any  character on $\Zun^+/\Gamma^+$ extends
 to a character of $  \ZG(\adef)/ \ZG(F)$. Denote by $\omega_1$ the character of $\Zun/\Gamma$
 defined by $\omega_1^+$.

\begin{remark}\label{contrex}
Observe that if $\f$ is injective, i.e.\! if $\H$ is a subgroup of $\G$, then 
$$\Gamma=\G(F)\cap\H(\adef)=\H(F)\,\,.$$
But when $\f$ is not injective it may happen that $\NN.\H(F)$ is a strict subgroup of $\Gamma$.  This is,
for example, the case if $\G=\GM_m$, $\H=\GM_m$
and $f:x\mapsto x^n$ when $(F,n)$ is a counter example to the local-global principle for $n$-th powers 
(see \cite[Chap. X, Thm. 1]{AT}).
\end{remark}

Since the group $\Zun$ normalizes $\H(F)$, it acts via left translations on $\H(F)\backslash \H(\adef)$, hence on the space 
$$L^2(\H(F)\backslash \H(\adef),\omega_0)$$ of functions that are
square-integrable modulo the center on $\H(F)\backslash \H(\adef)$ and that transforms according to $\omega_0$
some automorphic character of the center of $\H(\adef)$. 
The latter space is 
endowed with the right regular representation $\rho_{\omega_0}$ of $\H(\adef)$. 

The space of left $\Gamma$ invariant functions that are
square-integrable modulo the center on $\H(F)\backslash \H(\adef)$ can be decomposed
  according to the characters of $\Gamma\bs\Zun$ and this decomposition is compatible with the spectral 
decomposition of the right regular representation.  Observe that the action of $\Zun$ preserves cuspidality.
Now,  given $\omega$ and $\omega_1$ as above consider
a   function $\vf$ on $\H(\adef)$ which satisfies the condition
$$ \vf(c h) = \omega_1(c) \vf(h) \hbox{ for all}\;  c \in \Zun,\,\, h \in\H(\adef)\,\,.$$
There exists a unique function $\vf^+$ on $ \Hplus$ such that 
$$ {\vf}^+({z}\gamma g)  = \omega({z}) {\vf}^+(g)$$
 for any ${z} \in  \ZG(\adef)$,  $\gamma\in \G(F)$, $g\in\Hplus$, and moreover (using $\dot x$ to denote $\f(x)$)
 $${\vf}^+( \zeta \dot h) = \vf(c h)= \omega_1(c) \vf(h)$$
 whenever ${\zeta} = \dot c \gamma$ with $c \in \Zun$, ${\gamma} \in \G(F)$
 and  $h\in \H(\adef)$.
  This yields a bijection 
$$ L^2(\H(F)\backslash \H(\adef), \omega_1)
 \tilde{\longrightarrow} L^2(\G(F) \backslash \Hplus, \omega),$$
  that preserves cuspidality.   Here cuspidality for representations of $\Hplus$ has the obvious definition namely the vanishing
of integrals over quotients $U(F)\backslash U(\adef)$
of non trivial ``unipotent subgroups''  that are isomorphic images in $\G(F)\backslash\Hplus$ of quotients of unipotent subgroups  in $
\H(\adef)$.
  Hence one obtains a bijection between the cuspidal spectra
 $$ L^2_{\cusp}(\H(F)\backslash \H(\adef), \omega_1)
 \tilde{\longrightarrow}L^2_{\cusp}(\G(F) \backslash \Hplus, \omega).
 \leqno{(\star)} $$
  It is known that the right regular representation $\rho_{\cusp,\omega_1}$ of $\H(\adef)$ in
  $$ L^2_{\cusp}(\H(F)\backslash \H(\adef), \omega_1)$$ splits into a direct Hilbert sum with
  finite  multiplicities. This implies that
  the right regular representation of $\Hplus$ in
 $ \rho^+_{\cusp,\omega}$ in $L^2_{\cusp}(\G(F) \backslash \Hplus, \omega)$
 also  splits into a direct Hilbert sum with
  finite  multiplicities.

Now we observe that $L^2(\G(F) \backslash \Hplus, \omega)$ is the space of the representation $$\rho_{\omega}^+ = \text{Ind}
^{\Hplus}
    _{\Zplus}\omega,$$
while $L^2(\G(F) \backslash \G(\adef), \omega)$ is the space of the representation 
$$\rho_{\omega} = \text{Ind}^{\G(\adef)}
    _{ \Zplus}\,\,\omega.$$ Thus, since induction preserves cuspidality,    we see that
  $$   \rho_{\cusp,\omega} = \text{Ind}^{\G(\adef)}
    _{ \Hplus}\rho^+_{\cusp,\omega}.$$

\subsection{Main results}

\begin{theorem} \label{cuspa} The restriction to $\NN\bs\H(\adef)$ 
of any cuspidal representation $\pi$ of $\G(\adef)$
 contains a cuspidal representation $\sigma$ of $\H(\adef)$.
\end{theorem}

 \begin{proof}
  Any cuspidal automorphic representation $\pi$ of $\G(\adef)$
with central character $\omega$ occurs in 
$$\rho_{\sigma^+}=\text{Ind}^{\G(\adef)}_{ \Hplus}\sigma^+$$
for some  constituent $\sigma^+$  of $\rho_{\cusp,\omega}^+$. 
It follows from Lemma~\ref{finiadele} and  Proposition~\ref{frobenius} that $\sigma^+$ occurs in the restriction of $\pi$ to $\Hplus$.
  But the isomorphism $(\star)$ shows that
the restriction of $\sigma^+$ to $\NN\bs\H(\adef)$ is a direct sum of cuspidal representations. 
\end{proof}

\begin{theorem} \label{cuspb}
 Any cuspidal representation $\sigma$ of $\H(\adef)$ that can be realized in a space of functions on
 $\Gamma\bs\H(\adef)$
appears in the restriction of some cuspidal representation $\pi$ of $\G(\adef)$.
This is in particular true for any cuspidal representation of $\H(\adef)$ when $\f$ is injective.
\end{theorem}
 
\begin{proof} Consider  $\mathcal H$ the subspace 
of left $\Gamma$-invariant functions in the space of cuspidal square integrable functions modulo the center
$$L^2_{\cusp}(\H(F)\bs\H(\adef),\omega_0)$$
where $\omega_0$ is the character by which $\sigma$ acts when restricted to the center of $\H(\adef)$,
The space of the isotypic component, say $W_\sigma$, of $\sigma$ in $\mathcal H$ 
can be decomposed according to characters of $\Gamma\bs\Zun$ with $\Zun$ 
acting on the left.  Let $\omega_1$
be a character that occurs and consider the subspace $W_\sigma(\omega_1)$ of $W_\sigma$
cut out by this character. Choose $\omega$ extending $\omega_1$ to $\Zplus$. Then $W_{\sigma}(\omega_1)$
 can be mapped into a subspace of 
$\rho^+_{\cusp,\omega}$ via $(\star)$ and let $\sigma^+$ be  an irreducible constituent of the subspace
generated by the  image of $W_\sigma(\omega_1)$
under the action of $\Hplus$. 
We get a family of cuspidal representations for $\G$ by decomposing the induced representation 
$$\rho_{\sigma^+}=\text{Ind}^{\G(\adef)}_{ \Hplus}\sigma^+$$
According to \ref{cliffb} the various representations that occur in $\rho_{\sigma^+}$ are of the form
$\pi\otimes\chi$ for some $\pi$  where $\chi$  runs over characters of $\G(\adef)/\Hplus$. Thanks to Lemma~\ref{finiadele} and 
Proposition~\ref{frobenius} we know that $\sigma^+$ occurs in the restriction of $\pi$ to $\Hplus$ and, in turn, by construction, 
$\sigma$ occurs in the restriction of $\sigma^+$ to $\NN\bs\H(\adef)$.
The last statement follows from \ref{contrex}.
\end{proof}

\subsection{A reformulation}

\begin{definition}
We denote by $\mathcal{A}_{\G}(\H,\adef)$ the set of  $\G(\adef)$-conjugacy classes
of irreducible unitary representations of $\H(\adef)$ trivial on $\NN$.
\end{definition}

\begin{definition}

Two irreducible unitary representations $\pi $ and $\pi'$ of $\G(\adef)$ are said to be 
$\mathcal{E}_{\H}$-equivalent  if there exists a character $\mu$ of $\G(\adef)/\H(\adef)$ such that $\pi \otimes \mu \cong \pi'$.  
We denote by $\mathcal{E}_{\H}(\G, \adef)$ the  corresponding set of equivalence classes. 
\end{definition}

All elements in the $\mathcal{E}_{\H}$-class of some global $\pi$ have equivalent restrictions to $\H(\adef)$ and all components of 
the restriction
  belong to the same $\G$-packet. 
Let  $$R : \mathcal{E}_{\H}(\G, \adef) \to \mathcal{A}_{\G}(\H, \adef)$$  be the map
  which assigns to an $\mathcal{E}_{\H}$-equivalence class 
  represented by $\pi$ the $\G(\adef)$-packet   of components of $\pi\vert_{ \H(\adef)}$.   
  Observe that $R$ is  the  restricted product of  local restrictions. This makes sense since, for almost all places $v$,   the  restriction 
to $ \H_v$ of an 
unramified representation of $\G_v$
contains a unique  constituent that is unramified.

 \begin{proposition} \label{bijb}
The map
$$ R :\mathcal{E}_{\H}(\G, \adef) \to \mathcal{A}_{\G}(\H, \adef)$$
is a bijection. 
 \end{proposition}

\begin{proof}
 The local analogue \ref{fini}  implies the injectivity of $R$. The surjectivity follows from the
 local analogue and the
 fact that if $\G_v$ and $\H_v$ are unramified
 any unramified representation of  $\H_v$ occurs in the restriction of
 an unramified representation of $\G_v$.
 \end{proof}

We denote by  $\mathcal{A}_{\G,\cusp}(\H, \adef)$ the subset of $\mathcal{A}_{\G}(\H, \adef)$
 of $\G$-packets that contain
some cuspidal automorphic representation of $\H(\adef)$.

We define 
$\mathcal{E}_{\H,\mathrm{\cusp}}(\G, \adef)$ to be the subset of $\mathcal{E}_{\H}(\G, \adef)$
of $\mathcal{E}_{\H}$-equivalence classes 
that contain some cuspidal automorphic representations of $ \G(\adef)$.

\begin{theorem}\label{main} Assume that  $\Gamma=\NN.\H(F)$ (this is true
in particular when $\f$ is injective).
The map $$R: \mathcal{E}_{\H}(\G, \adef) \to \mathcal{A}_{\G}(\H, \adef)$$
induces  a bijection 
$$\mathcal{E}_{\H,\cusp}(\G, \adef) \tilde{\longrightarrow} \mathcal{A}_{\G,\cusp}(\H, \adef)\,\,.$$
\end{theorem}
\begin{proof} In view of Propositions \ref{bija}, \ref{bijb} and  Remark~\ref{contrex} 
this is nothing but
a reformulation of Theorems \ref{cuspa} and \ref{cuspb}.
\end{proof}

Observe that when  $\Gamma$ is strictly bigger than $\NN.\H(F)$ (in particular $\f$ is not injective)
the map 
$$\mathcal{E}_{\H,\cusp}(\G, \adef) \to \mathcal{A}_{\G,\cusp}(\H, \adef)$$
may not be surjective: an example is given in Remark~\ref{contrex}.
The image consists of classes of cuspidal representations that can be realized
in a subspace of $\Gamma$-left-invariant functions.

\begin{remarks} The reader should be aware of the following pitfalls.
\pni 1 - If $\sigma$ is a cuspidal representation of $\H(\adef)$ 
 it is not always the case that all  conjugates $\sigma^g$ for $g\in\G(\adef)$
are automorphic. Examples of this  fact do occur in the case $\H=SL(n)$ and $\G=GL(n)$
for representations that are ``endoscopic'' (see~\cite{LL} for the case $n=2$).
\pni 2 - Consider two cuspidal automorphic representations $\pi$ and $\pi'$ that
are  of the form $\pi'\simeq\pi\otimes\mu$;  it may happen that
$\mu$  cannot be chosen to be automorphic (see \cite{BL} where examples are constructed
for $\H=SL(n)$ and $\G=GL(n)$ provided $n\ge 3$).
\end{remarks}

\subsection{A multiplicity formula}\label{fail}
We assume moreover from now on that $\f$ is injective.
 Given an irreducible unitary representation $\pi$ of $\G(\adef)$ the restriction of  $\pi$ to $\H(\adef)$ splits into a direct sum with
finite multiplicities if $\pi_v$ is generic almost everywhere.  
In fact the  restriction to $ \H_v$ of an 
unramified representation 
contains a unique  constituent that is unramified.
The representation $\pi\vert_{\H(\adef)}$ is the direct sum of the restricted products of the
constituents of  the $\pi_v\vert_{\H_v}$.  
We know that locally everywhere the multiplicity is finite
(cf.~\ref{fini}).
But, whenever $\pi_v$ has a Whittaker model, the restriction
is multiplicity free.  Hence the global decomposition is a direct sum (infinite in general) and
with finite multiplicities if $\pi_v$ is generic almost everywhere.

We observe that given $\pi$ the set components of $\pi\vert_{\Hplus}$ is finite according to
Propositions \ref{cliffa} and \ref{finiadele}, but one should be aware that not all such representations
 will show up in $\rho^+_{\cusp,\omega}$. In fact, for example,
if $\G=GL(n)$ only one such $\sigma^+$, in the restriction to $\Hplus$ of a given $\pi$,
may occur  in $\rho^+_{\cusp,\omega}$ since otherwise
this would contradict the multiplicity one theorem for  cuspidal representations of $GL(n)$. 
On the other hand there may be more than one  $\sigma^+$ in the space  generated  by the isotypic component
of some $\sigma$ and they may be inequivalent. 
This is in fact the case when considering cuspidal representations of $SL(n)$ with multiplicity
greater than one (which may exist for $n\ge3$).  In such a case the various $\pi$'s  
containing $\sigma$ in their restriction to $\H(\adef)$ may differ by
non automorphic characters (see \cite{BL}).
More generally we have the following multiplicity formula.

\begin{theorem}\label{multi}
Assume $\G$ and $\H$ quasi-split.
Let $\pi$ be a generic cuspidal representation for $\G$
and $\sigma$ a generic cuspidal representation for $\H$ that
occurs in the restriction of $\pi$ to $\H(\adef)$.
Let $Y(\pi)$ be the group of characters $\mu$ of $\G(\adef)/\ZG(\adef)\H(\adef)$
such that $\pi\otimes\mu$ is also a  cuspidal representation. Let
$X_{loc}(\pi)$ the subgroup of characters
$\mu\in Y(\pi)$ such that $\pi\otimes\mu\simeq\pi$.
This is the restricted product over the set of places of $F$ of the $X(\pi_v)$.
Let $m(\pi)$ be the multiplicity of $\pi$ in the cuspidal spectrum for $\G$.
Then,  the multiplicity  $m(\sigma)$ of $\sigma$ in the cuspidal spectrum of $\H$ is given by
$$m(\sigma)=\sum_{\mu\in M(\pi)}  m(\pi\otimes\mu)$$
where $M(\pi)=Y(\pi)/X_{loc}(\pi).X$.

\end{theorem}

\begin{proof}  The uniqueness of Whittaker models tells us that the restriction of $\pi$
to $\H(\adef)$ is multiplicity free. In particular any $\pi$ defines a unique $\sigma^+$
in $\rho^+_{\cusp,\omega}$ and conversely this $\sigma^+$ is associated to
the set of cuspidal representations of the form $\pi\otimes\chi$ with $\chi\in X$ i.e.\!
trivial on $\Hplus$, in particular $\chi$ is automorphic. Now the set of representations
$\pi'$ in $ \rho_{\cusp,\omega}$ whose restriction to $\H(\adef)$ contains $\sigma$,
is the set of $\pi'=\pi\otimes\mu$ with $\mu\in Y(\pi)$.
\end{proof}

\subsection{Miscellaneous remarks}
Assume again $\f$ injective.
 Let $\ZO= \ZG(\adef)\cap \H(\adef)$.
 The group $$ \Zun=\ZG(\adef) \G(F) \cap \H(\adef)$$ is often equal to 
$\ZO.\H(F)$. For example, this  latter equality holds in the case  $\G=GL(n)$ and $\H=SL(n)$
whenever the local-global principle  for $n^{th}$-roots of unity holds for $F$ and $n$.  In fact, if $z\gamma\in \Zun$ which means
$\text{det}( z\gamma)=1$ then $\text{det}(\gamma)$ is locally everywhere an $n^{th}$-power, and, if the
local-global principle holds, this means that $\text{det}(\gamma)$ is itself an $n^{th}$-power and $\gamma$ can be rewritten as
$\zeta.\eta$ with $\zeta\in  \ZG(F)$ and $\eta\in \H(F)$, hence $z\gamma=z_0\eta$ with $z_0\in \ZO$.  This shows that, in this case,
the new argument is essentially identical to the argument used in \cite[Sect. 3]{LS}.

Transfer results similar to Theorem~\ref{cuspa} and Theorem~\ref{cuspb}
for  cuspidal automorphic forms, have been obtained by Chenevier \cite{Ch} 
under the condition $\Zun=\ZO.\H(F)$.

As observed in the introduction the map $f:\H\to G$ induces a map $\check f$ between
dual groups which in turn defines a map 
$$H^1(\check f):H^1(W_F,\check\G)\to H^1(W_F,\check\H)$$
between cohomology sets for Weil groups with values in dual groups.
The  induced correspondence between packets of
representations for $\H$ to packets for  $\G$ 
provided by Langlands functoriality conjectures should fit  with
\ref{local} and \ref{cuspa}.
If $f$ is injective Proposition \ref{bija}
 and Theorem~\ref{main} are compatible with lifting Theorems 7.1 and  8.1 in \cite{Lab}.

\appendix

\def\ni{\noindent}

\section{Central Morphisms}

\centerline{By Bertrand LEMAIRE}

\bigskip

Let $k$ be a field.
Recall that a morphism of algebraic groups $f: H\rightarrow G$ (over $k$)  is said to be {\it central} if the schematic kernel of $f$ is 
contained 
in the schematic center of $H$, which means that for any commutative $k$-algebra $A$, we have the inclusion
$$
\Ker \left( f_A: H(A) \rightarrow G(A)\right)  \subset Z(H(A)),
$$
where $Z(H(A))$ denotes the center of $H(A)$.\footnote{%
Note that if $k$ is of caracteristic $p>0$, a surjective central $k$-morphism 
(e.g. a central $k$-isogeny) may  be inseparable: for example, the map $t\mapsto t^2$ from the multiplicative group 
${\mathbb G}_m$ into itself, with $p=2$.}
From \cite[22.4]{B} we know that, given a connected reductive $k$-group $G$, the product morphism
$
Z^G\times G_{der} \rightarrow G
$
 where $Z^G$ and $G_{der}$ are respectively the (set-theoretic) center and the derived group of $G$, is a central $k$-isogeny.

Let $F$ be  a global field.  We denote by $\adef$ the adle ring of $F$. If $v$ is a place of $F$ its completion $F_v$ is
either $\RM$ or $\CM$ or
a non-Archimedean local field (i.e.\! a finite extension of ${\mathbb Q}_p$, resp. ${\mathbb F}_p((t))$).

\subsection{Surjective maps of tori}
\begin{lemma}\label{apptor}
Let $f: T\rightarrow S$ be a surjective morphism of tori.
\begin{enumerate}
\item[(1)] For any place $v$, the group $S(F_v)/f(T(F_v))$ is compact.
\item[(2)] The group $S(\adef)/f(T(\adef))S(F)$ is compact.
\end{enumerate}
\end{lemma}

\begin{proof}
We only give a proof for assertion (2), the proof of (1) being essentially the same but simpler. Let $S_d$ be the maximal $F$-split 
subtorus of $S$, and $X(S_d)$ the group of algebraic characters of $S$ (they are all defined over $F$). Let us fix a finite place 
$v$ of $F$, and a uniformizer element $\varpi_v$ 
of the completion $F_v$ of $F$ at $v$. The set
$$
S(\varpi_v)=\Hom (X(S),\varpi_v^{\mathbb Z})
$$
is a free abelian group of finite rank, and a co-compact subgroup of $S_d(F_v)$. It also naturally identifies with a subgroup of 
$S_d(\adef)$. Moreover $S_d(\varpi_v)\cap S_d(F)=\{1\}$ and the group $S_d(\adef)/S_d(\varpi_v)S_d(F)$ is compact. Now let 
$\overline{S}=S/S_d$. It is an $F$-anisotropic torus, hence the group $\overline{S}(\adef)/\overline{S}(F)$ is compact \cite[3.5]
{Sp}, which implies the group $S(\adef)/S_d(\adef)S(F)$ is compact. Since
$$
S_d(\adef) \cap (S_d(\varpi_v)S(F))= S_d(\varpi_v)S_d(F),
$$
we obtain the group $S(\adef)/S_d(\varpi_v)S(F)$ is compact. On the other hand, $f$ induces a surjective $F$-morphism $f_d: T_d 
\rightarrow S_d$ 
which sends $T_d(\varpi_v)$ onto a sub-lattice of $S_d(\varpi_v)$. Hence the group $S(\adef)/f(T_d(\varpi_v))S(F)$ is compact. This 
implies (2). 
\end{proof}

\subsection{Central surjective morphisms of reductive groups}
\begin{proposition}\label{appcomp}
Let $f:H \rightarrow G$ be a surjective central morphism of connected reductive groups. 
\begin{enumerate}
\item[(1)] For any place $v$, the quotient $G(F_v)/f(H(F_v))$ is an abelian compact group.
\item[(2)] The quotient $G(\adef)/ f(H(\adef))G(F)$ is an abelian compact group. 
\end{enumerate}
\end{proposition}

\begin{proof}
As above, we only give a proof of assertion (2). From \cite[2.2, 2.6]{BT} (see \cite[2.3]{BT}), there exists an 
$F$-morphism $\kappa:G \times G \rightarrow H$ such that for all $x,\,y\in H$, we have
$$
\kappa(f(x),f(y))= xyx^{-1}y^{-1}.
$$
So the commutator map $G\times G\rightarrow G,\, (x,y)\mapsto [x,y]= xyx^{-1}y^{-1}$ coincides with $f\circ \kappa$, and we have
$$
[G(\adef),G(\adef)]= f\circ \kappa (G(\adef)\times G(\adef))\subset f(H(\adef)).
$$
Hence $f(H(\adef))$ is an invariant subgroup of $G(\adef)$, and the group $G(\adef)/ f(H(\adef))$ is abelian. 
A fortiori the quotient $G(\adef)/f(H(\adef))G(F)$ is an abelian group. It remains to prove the compacity. 
Let $S$ be a maximal $F$-split torus in $G$, and $M=Z^G(S)$ the centralizer of $S$ in $G$. 
Let $P$ be a (minimal) parabolic $F$-subgroup of $G$ with Levi component $M$, and $U=U_P$ the unipotent radical of $P$. 
From \cite[22.6]{B}, the inverse image $S'$ of $S$ in $H$ is a maximal $F$-split torus in $H$, and the inverse image $P'$ of $P$ in 
$H$ is a minimal parabolic $F$-subgroup of $H$. Put $M'=Z^H(S')$ and $U'=U_{P'}$. From loc.~cit., $f$ induces a surjective 
$F$-morphism  
$M'\rightarrow M$ and an $F$-isomorphism $U'\rightarrow U$. Moreover, $M'\rightarrow M$ is central. On the other hand, we have 
the Iwasawa decomposition
$
G(\adef)= \boldsymbol{K}P(\adef)
$
where $\boldsymbol{K}= \prod_v\boldsymbol{K}_v$ is an $M$-admissible  maximal compact subgroup of $G(\adef)$. Hence the 
product map 
$\boldsymbol{K}\times P(\adef)\rightarrow G(\adef)$ gives a surjective map
$$
\boldsymbol{K} \times (P(\adef)/f(P'(\adef))P(F) \rightarrow G(\adef)/ f(H(\adef))G(F).
$$
Since $f(U'(\adef))=U(\adef)$, we have
$$P(\adef)/f(P'(\adef))P(F)= M(\adef)/f(M'(\adef))M(F).$$
So we just need to prove the compacity of the quotient (that we already know to be an abelian group) $M(\adef)/f(M'(\adef))M(F)$.
Since the quotient $\overline{M}= M/S$ is a connected reductive $F$-anisotropic group, the set $\overline{M}(\adef)/\overline{M}(F)$ 
is compact 
\cite[3.5]{Sp}, which implies the set $M(\adef)/ S(\adef)M(F)$ is compact. A fortiori the quotient 
$$M(\adef)/f(M'(\adef))S(\adef)M(F)$$
is compact. Since
$$f(M'(\adef))\cap (S(\adef)M(F))= f(S'(\adef))S(F),$$
we are reduced to prove the compacity of the group $S(\adef)/ f(S'(\adef))S(F)$. It is given by the Lemma~\ref{apptor}.
\end{proof}

\begin{corollary}\label{appcompb}
Let $f:H\rightarrow G$ be an $F$-morphism of connected reductive groups such that the induced morphism 
$f_{der}: H_{der} \rightarrow G_{der}$ is a central isogeny.
\begin{enumerate} 
\item[(1)]For any place $v$, the quotient $G(F_v)/Z^G(F_v)f(H(F_v))$ is an abelian compact group.
\item[(2)]The quotient $G(\adef)/Z^G(\adef)f(H(\adef))G(F)$ is an abelian compact group.
\end{enumerate}
\end{corollary}

\begin{proof}
The morphism 
$$\textrm{id} \times f_{der} : Z^G \times H_{der} \rightarrow Z^G \times G_{der}$$
and the product morphism
$Z^G \times G_{der} \rightarrow G$
are central $F$-isogenies. The composition of these two morphisms 
$Z^G \times H_{der} \rightarrow  G$
is also a central $F$-isogeny. This implies the corollary.
\end{proof}

\begin{remarks}
In the corollary, we may replace $Z^G$ by its connected component $Z$, which is the maximal central $F$-torus in $G$; the product 
morphism $Z\times G_{der} \rightarrow G$ is still a central $F$-isogeny. 
\end{remarks}

\bibliographystyle{amsalpha}

\end{document}